\documentclass{amsart}
\usepackage{amsmath}
\usepackage{amssymb}
\usepackage{boxedminipage}
\usepackage[arrow,matrix,curve]{xy}
\newcommand{\co}{\colon\,}

\newcommand{\bR}{\mathbb R}
\newcommand{\bC}{\mathbb C}

\newcommand{\bK}{\mathbb K}
\newcommand{\bL}{\mathbb L}
\newcommand{\bZ}{\mathbb Z}
\newcommand{\bP}{\mathbb P}

\newcommand{\bN}{\mathbb N}
\newcommand{\bQ}{\mathbb Q}

\newcommand{\wcA}{\widehat{\mathcal A}}

\newcommand{\cE}{\mathcal E}
\newcommand{\cF}{\mathcal F}

\newcommand{\cL}{\mathcal L}
\newcommand{\cN}{\mathcal N}
\newcommand{\cO}{\mathcal O}

\newcommand{\cS}{\mathcal S}
\newcommand{\cT}{\mathcal T}
\newcommand{\wM}{\widetilde M}
\newcommand{\pt}{\text{pt}}
\newcommand{\lp}{\textup{(}}
\newcommand{\rp}{\textup{)}}

\newcommand{\Wh}{\operatorname{Wh}}

\newcommand{\sign}{\operatorname{sign}}
\newcommand{\inde}{\operatorname{ind}}

\newcommand{\id}{\text{id}}
\theoremstyle{plain}
\newtheorem{theorem}{Theorem}[section]

\newtheorem{proposition}[theorem]{Proposition}
\newtheorem{conj}[theorem]{Conjecture}
\theoremstyle{definition}

\newtheorem{definition}[theorem]{Definition}
\newtheorem{problem}[theorem]{Problem}

\newlength{\boxwidth}
\begin{document}

\title{Novikov's Conjecture}
\author{Jonathan Rosenberg}
\address{Department of Mathematics,
  University of Maryland, College Park, MD 20742-4015, USA}
\email{jmr@math.umd.edu}
\thanks{Work on this paper was partially supported by the United States National
Science Foundation, grant number 1206159. I would like to thank Andrew
Ranicki and Shmuel Weinberger for useful feedback on an earlier draft
of this paper.}
\begin{abstract}
  {We describe Novikov's ``higher signature conjecture,'' which
dates back to the late 1960's, as well as many alternative
formulations and related problems.  The Novikov Conjecture is perhaps
the most important unsolved problem in high-dimensional manifold
topology, but more importantly, variants and analogues permeate many
other areas of mathematics, from geometry to operator algebras to
representation theory.}
\end{abstract}
\subjclass[2010]{Primary 57R67; Secondary 19J25, 19K56, 19G24, 19D50, 58J22}
\maketitle

\section{Origins of the Original Conjecture}
\label{sec:origins}

The Novikov Conjecture is perhaps the most important unsolved problem
in the topology of high-dimensional manifolds.  It was first stated by
Sergei Novikov, in various forms, in his lectures at the International
Congresses of Mathematicians in Moscow in 1966 and in Nice in 1970,
and in a few other papers
\cite{MR0231401,MR0431215,MR0268907,MR0292913}. For an annotated
version of the original formulation, in both Russian and English, we
refer the reader to \cite{MR1388295}.  Here we will try instead to put
the problem in context and explain why it might be of interest to the
average mathematician.  For a nice book-length exposition of this
subject, we recommend \cite{MR2117411}.  Many treatments of various
aspects of the problem can also be found in the many papers in
the collections \cite{MR1388294,MR1388306}.

For the typical mathematician, the most important topological spaces
are smooth manifolds, which were introduced by Riemann in the 1850's.
However, it took about 100 years for the tools for classifying
manifolds (except in dimension $1$, which is trivial, and dimension
$2$, which is relatively easy) to be developed.  The problem is that
manifolds have no local invariants (except for the dimension); all
manifolds of the same dimension look the same \emph{locally}. Certainly many
different manifolds were known, but how can one tell whether or not
the known examples are ``typical''?  How can one distinguish one
manifold from another?

With big leaps forward in topology in the 1950's, it finally became
possible to answer these questions, at least in part.  Here were a few
critical ingredients:
\begin{enumerate}
  \item the development of the theory of Reidemeister and Whitehead
    torsion and the related notion of ``simple homotopy equivalence''
    (see \cite{MR0196736} for a good survey of all of this);
  \item the theory of characteristic classes of vector bundles,
    developed by Chern, Weil, Pontrjagin, and others;
  \item the notion of cobordism, introduced by Thom \cite{MR0061823},
    who also provided a method for computing it;
  \item the Hirzebruch signature theorem $\sign(M) = \langle \cL(M),
    [M]\rangle$ \cite{MR0368023}, giving a
    formula for the signature of an oriented closed manifold $M^{4k}$ (this is
    the algebraic signature of the nondegenerate symmetric bilinear
    form $(x,y)\mapsto \langle x\cup y, [M]\rangle$ on
    $H^{2k}$ coming from Poincar\'e duality), in terms of a certain
    polynomial $\cL(M)$ in the rational Pontrjagin classes of the
    tangent bundle.  
\end{enumerate}

Using just these ingredients, Milnor \cite{MR0082103} was able to
show that there are at least $7$ different diffeomorphism classes
of $7$-manifolds homotopy equivalent to $S^7$.  (Actually there are
$28$ diffeomorphism classes of such manifolds, as Milnor and Kervaire
\cite{MR0148075} showed a bit later.) This and the major role played by
items 2 and 4 on the above list\footnote{Spheres have stably trivial
tangent bundle and no interesting cohomology, so one's first guess
might be that the theory of vector bundles and the signature theorem
might be irrelevant to studying homotopy spheres.  Milnor, however,
showed that one can construct lots of manifolds with the homotopy type
of a $7$-sphere as unit sphere bundles in rank-$4$ vector bundles over
$S^4$.  He also showed that the signature of an $8$-manifold 
bounded by such a manifold yields lots of information
about the homotopy sphere.} came as a big surprise, and
showed that the classification of manifolds, even within a
``standard'' homotopy type, has to be a hard problem.

The final two ingredients came just a bit later.  One was Smale's
famous \emph{$h$-cobordism theorem}, which was the main ingredient in his
proof \cite{MR0137124} of the high-dimensional Poincar\'e conjecture
in the topological category.  (In other words, if $M^n$ is a smooth
compact $n$-manifold, $n\ge 5$, homotopy equivalent to $S^n$, then $M$
is homeomorphic to $S^n$, even though it may not be diffeomorphic to
it.) But from the point of view of the general manifold classification
program, Smale's important contribution was a criterion for telling
when two manifolds really are diffeomorphic to one another. An
$h$-cobordism between compact manifolds $M$ and $M'$ is a compact
manifold with boundary $W$, such that $\partial W = M\sqcup M'$ and
such that $W$ has deformation retractions down to both $M$ and
$M'$. The $h$-cobordism theorem \cite{MR0190942}
says that if $\dim M=\dim M' \ge 5$
and if $M$, $M'$, and $W$ are simply connected, then $W$ is
diffeomorphic to $M\times [0,1]$, and in particular, $M$ and $M'$ are
diffeomorphic. The advantage of this is that diffeomorphisms between
different manifolds are usually very hard to construct directly; it is
much easier to construct an $h$-cobordism.

If one dispenses with simple connectivity, then an $h$-cobordism
between $M$ and $M'$ need not be diffeomorphic to a product $M\times
[0,1]$. However, the $s$-cobordism theorem, due to Barden, Mazur, and
Stallings, with simplifications due to Kervaire \cite{MR0189048},
says that the $h$-cobordisms themselves are classifiable by the
Whitehead torsion $\tau(W, M)$, which takes values in the Whitehead
group $\Wh(\pi)$, where $\pi=\pi_1(M)$, and all values in $\Wh(\pi)$
can be realized by $h$-cobordisms. (The Whitehead group is the
quotient of the algebraic $K$-group $K_1(\bZ\pi)$ by its ``obvious''
subgroup $\{\pm1\}\times \pi_{\text{ab}}$.)  Thus an $h$-cobordism is a product if
$\Wh(\pi) = 0$, which is the case for $\pi$ free abelian, and in fact
is conjectured to be the case if $\pi$ is torsion-free.  But for $\pi$
finite, for example,
$\Wh(\pi)$ is a finitely generated group of rank $r-q$, where
$r$ is the number of irreducible real representations of $\pi$, and
$q$ is the number of irreducible rational representations of $\pi$
\cite[Theorem 6.2]{MR0196736}.  This number $r-q$ is usually positive
(for example, when $\pi$ is finite cyclic, it vanishes only if $|\pi|=
1,2, 3, 4, \text{ or }6$).  Bass and Murthy have even shown
\cite{MR0219592} that there are finitely generated abelian groups
$\pi$ for which $\Wh(\pi)$ is not finitely generated.
 
The last major ingredient for the classification of manifolds is the
method of \emph{surgery}.  Surgery on an $n$-manifold $M^n$ means cutting
out a neighborhood $S^k\times D^{n-k}$ of a $k$-sphere
$S^k\hookrightarrow M$ (with trivial normal bundle)
and replacing it by $D^{k+1}\times S^{n-k-1}$,
which has the same boundary.  This can be used to modify a manifold
without changing its bordism class, and was first introduced by Milnor
\cite{MR0130696} and Wallace \cite{MR0125588}.

With the help of all of these techniques, Browder
\cite{MR0326743,MR0358813} and Novikov \cite{MR0158405,MR0162246}
finally introduced a general methodology for classifying manifolds in high
dimensions. The method gave complete results for simply connected
manifolds in dimensions $\ge 5$,
and only partial information in dimensions $3$ and $4$, which have
their own peculiarities we won't discuss here.  With the help of
additional contributions by Sullivan \cite{MR1434103},
Novikov  \cite{MR0292913}, 
and above all, Wall \cite{MR0212827}, this method
grew into what we know today as \emph{surgery theory}, codified by
Wall in his book \cite{MR1687388}, which originally appeared in 1970.
There are now fairly good expositions of the theory, for example in
Ranicki's books \cite{MR1211640,MR2061749}, in the book by Kreck and
L\"uck \cite{MR2117411}, in the first half of Weinberger's book
\cite{MR1308714}, and in Browder's colloquium lectures from 1977
\cite{MR1747530}, so we won't attempt to compete by going into
details, which anyway would take far too many pages.  Instead we will
just outline enough of the ideas to set the stage for Novikov's conjecture.

As we indicated before, surgery theory addresses the uniqueness question for
manifolds: given (closed and connected, say) manifolds $M$ and $M'$ of the same
dimension $n$, when are they diffeomorphic (or homeomorphic)?  It also
addresses an existence question: given a connected topological space $X$ (say a
finite CW complex), when is it homotopy equivalent to a (closed) manifold?

A few necessary conditions are evident from a first course in topology.
If $M$ and $M'$ are diffeomorphic, then certainly they are homotopy
equivalent, and so they have the same fundamental group $\pi$.
Furthermore, if a finite connected CW complex $X$ has the homotopy
type of a closed manifold, then it has to satisfy Poincar\'e duality,
even in the strong sense of (possibly twisted) Poincar\'e duality
of the universal cover with coefficients in $\bZ\pi$.  Homotopy
equivalences preserve homology and cohomology groups and cup products,
so an orientation-preserving homotopy equivalence also preserves the
signature (in dimensions divisible by $4$ when the signature is defined).
However, these conditions are not
nearly enough.  For one thing, for a homotopy equivalence to be
homotopic to a diffeomorphism (or even a homeomorphism), it has to
be \emph{simple}, i.e., to
have vanishing torsion in $\Wh(\pi)$.  Depending on the fundamental
group $\pi$, this may or may not be a serious restriction.

But the most serious conditions involve characteristic classes of the
tangent bundle.  Via a very ingenious argument using surgery theory
and the Hirzebruch signature theorem, Novikov
\cite{MR0193644,MR0196764} showed that the rational Pontrjagin 
classes of the tangent bundle of a manifold are preserved under
homeomorphisms.\footnote{The same does not hold for the torsion part
of the Pontrjagin classes, as one can see from calculations with lens
spaces \cite[\S3]{MR0268907}.} 
(Incidentally, Gromov \cite[\S7]{MR1389019} has given a totally 
different short argument for this.) The rational Pontrjagin classes do
\emph{not} 
have to be preserved under homotopy equivalences. So if $\varphi\co
M\to M'$ is a homotopy equivalence not preserving rational Pontrjagin
classes, it cannot be homotopic to a homeomorphism.

In the simply connected case, this is (modulo finite ambiguity)
just about all: if $M'\to M$ is an orientation-preserving homotopy
equivalence of closed (oriented) simply connected oriented manifolds, 
the rational Pontrjagin classes of $M'$ have to satisfy the
constraint $\langle \cL(M'), [M']\rangle = \sign(M')=\sign(M)$
imposed by the Hirzebruch signature theorem, but otherwise
they are effectively unconstrained (assuming the dimension of the
manifold is at least $5$).\footnote{A precise statement to this effect
may be found in \cite[Theorem 6.5]{MR1747536}.  It says for example
that if $M$ is  closed simply connected manifold and $\dim M$ is not
divisible by $4$, then for \emph{any} set of elements $x_j\in H^{4j}(M,\bQ)$,
$1\le j\le \left\lfloor \frac{\dim M}{4}\right\rfloor$, there is a
positive integer $R$ such that for any integer $m$, there is
a homotopy equivalence of manifolds $\varphi_m\co M'_m\to M$ such
that $p_j(M'_m) = \varphi_m^*\bigl(p_j(M)+ m\,R\,x_j\bigr)$.}  
And if the map does preserve rational
Pontrjagin classes, then there are only finitely many possibilities
for $M'$ up to diffeomorphism.

When $M$ is not simply connected, the situation is appreciably more
complicated. Suppose one wants to check if two $n$-manifolds $M$ and $M'$
are diffeomorphic.  As we indicated before, that means we need to have
a simple homotopy equivalence $\varphi\co M' \to M$. If $\varphi$ were
homotopic to a diffeomorphism, it would preserve the classes of the
tangent bundles, so it's convenient to assume that $\varphi$ has been
promoted to a \emph{normal map} $\varphi\co (M', \nu')\to (M,
\nu)$. Here $\nu$ and $\nu'$ are the stable normal bundles defined via
the Whitney embedding theorem: if $k$ is large enough ($n + 1$
suffices), then $M$ and $M'$ have embeddings into Euclidean space
$\bR^{n+k}$, and any two such embeddings are isotopic, so the 
isomorphism class of the normal bundle $\nu$ or $\nu'$ for such an
embedding is well defined. (Because of the Thom-Pontrjagin
construction, it's better to work with the normal bundle than with the
tangent bundle, but they contain the same information.)  Being a
normal map means that $\varphi$ has been extended to a bundle map from
$\nu'$ to $\nu$, which we can assume is an isomorphism of bundles.
The idea of trying to show that $M$ and $M'$ are diffeomorphic is to
start with a \emph{normal bordism} from $\varphi$ to $\id_M$, i.e., a
manifold $W^{n+1}$ with boundary $M\sqcup M'$ and a map $\Phi\co W\to
M\times [0,1]$ restricting to $\varphi$ and to $\id_M$ on the two boundary
components, and with  a compatible map of bundles, and then to try to
modify $(W,\Phi)$ by surgery to make it into an $s$-cobordism.  Once this is
accomplished, then $M$ and $M'$ are diffeomorphic by the $s$-cobordism
theorem. It turns out that doing the surgery is not difficult until
one gets up to the middle dimension (if $n+1$ is even) or the ``almost
middle'' dimension $\left\lfloor \frac{n+1}{2} \right\rfloor$ (if
$n+1$ is odd).  At this point a \emph{surgery obstruction} appears,
taking its value in a group $L_{n+1}(\bZ\pi)$ constructed purely
algebraically out of quadratic forms on $\bZ\pi$. 
(Roughly speaking, the $L$-groups are groups of stable equivalence
classes of forms on finitely generated projective or free $\bZ\pi$-modules,
and the type of the form --- symmetric,
skew-symmetric, etc. --- depends only on the value of $n$ mod $4$.
The original construction may be found in \cite{MR1687388}.)
The existence problem (telling if one can find a manifold
homotopy equivalent to a given finite complex with Poincar\'e duality)
works in a very similar way, just down in dimension by $1$, and the
surgery obstruction in that case takes its values in $L_{n}(\bZ\pi)$.

Ultimately, the result of this surgery process is to prove that there
is a \emph{surgery exact sequence} for computation of the
\emph{structure set} $\cS(M)$, the set of (simple) homotopy
equivalences $\varphi\co M'\to M$, where $M'$ is a smooth compact manifold,
modulo equivalence.  We say that two such maps $\varphi\co M'\to M$
and $\varphi'\co M''\to M$ are equivalent if there is a commuting
diagram
\[
\xymatrix{M' \ar[rr]^\varphi \ar[rd]_\cong&& M\\
& M'' \ar[ru]^{\varphi'}&.}
\]
The surgery exact sequence then takes the form
\begin{equation}
\label{eq:surgseq}
\xymatrix{\cdots \ar[r]^(.4)\alpha &
L_{n+1}(\bZ\pi)\ar@{.>}[r] & \cS(M) \ar[r]^\eta &\cN(M) 
\ar[r]^\alpha &L_{n}(\bZ\pi)}.
\end{equation}
Here $\cN(M)$ is the set of \emph{normal invariants}, the normal
bordism classes of all normal maps $\varphi\co (M',\nu')\to (M,\nu)$
(not necessarily homotopy equivalences as before)
modulo linear automorphisms of $\nu$. This can also be identified with
homotopy classes of maps from $M$ into a classifying space called
$G/O$.  If one works instead in the PL or the topological
category, the same sequence
\eqref{eq:surgseq} is valid, but $G/O$ is replaced by $G/PL$ or
$G/\mathit{Top}$, which are 
easier to deal with\footnote{once the dimension is bigger than $4$!},
and in fact look a lot like $BO$, the 
classifying space for real $K$-theory.  The natural maps $G/O\to
G/PL\to  G/\mathit{Top}$
are rational homotopy equivalences.  The map $\eta\co \cS(M) \to
\cN(M) $ sends a homotopy equivalence $\varphi\co M'\to M$ to the
associated normal data.

The groups $L_\bullet(\bZ\pi)$
are $4$-periodic, and only depend on the fundamental group and some
``decorations'' which we are suppressing here, which only affect
the torsion.  The map $\eta\co \cN(M) \to L_n(\bZ\pi)$ takes the
bordism class of a normal map $\varphi\co (M',\nu')\to (M,\nu)$
to its associated \emph{surgery obstruction}. When this vanishes,
exactness  of \eqref{eq:surgseq} says we can lift $\varphi$ to an
element of $\cS(M)$, or in other words, we can do surgery to convert
it to a homotopy equivalence.  The dotted arrow from $L_{n+1}(\bZ\pi)$
to $\cS(M)$ signifies that the surgery group operates on $\cS(M)$
(which is just a pointed set, not a group) and that two elements of
the structure set have the same normal invariant if and only if they
lie in the same orbit for the action of $L_{n+1}(\bZ\pi)$.

The exact sequence \eqref{eq:surgseq} is closely related to an
\emph{algebraic surgery exact sequence}
\begin{equation}
\label{eq:algsurgseq}
\cdots \to L_{n+1}(\bZ\pi) \to \cS_n(M) \to H_n(M, \bL(\bZ))
\xrightarrow{A} L_n(\bZ\pi) 
\end{equation}
constructed in \cite{MR561227,MR1211640}, where the map $A$, called
the \emph{assembly map}, corresponds to local-to-global passage. We
will come back to this later.

For most groups $\pi$, the $L$-groups $L_\bullet(\bZ\pi)$ are not easy
to calculate, so a lot of the literature on surgery theory emphasizes
things related to the exact sequence \eqref{eq:surgseq} which don't rely on
explicit calculation of all the groups.  For example, sometimes one
can compare two related surgery problems, or rely on other invariants,
such as $\eta$- and $\rho$-invariants for finite groups.  These (as
well as direct calculation from \eqref{eq:surgseq}) show that there
are infinitely many manifolds with the homotopy type of
$\bR\bP^{4k+3}$, $k\ge 1$. In fact, it's shown in \cite{MR1988288}
that in dimension ${4k+3}$, $k\ge 1$, any closed manifold $M$ with
torsion in its fundamental group has infinitely many distinct
manifolds simple homotopy-equivalent to it.

Now we are ready to explain Novikov's conjecture.  For $M$ an oriented
closed manifold, we can rewrite the Hirzebruch signature theorem as
saying that for a closed connected oriented manifold $M$, the $0$-degree
component of $\cL(M)\cap [M]$ in $H_0(M,\bQ)\cong \bQ$ coincides with
$\sign M$, which is preserved by orientation-preserving homotopy
equivalences.  The components of $\cL(M)\cap [M]$ in other degrees
have no such invariance property, and knowing them is equivalent to
knowing the rational Pontrjagin classes.  However, Novikov
discovered in \cite{MR0196764} (see \cite[Theorem 2.1 and its
proof]{MR1747536} for a simplified version of his argument) that
if $\pi_1(M)\cong \bZ$, then the degree-$1$ component of $\cL(M)\cap
[M]$ is also an oriented homotopy invariant.  This theorem is the
simplest special case of Novikov's conjecture. 
\begin{definition}
\label{def:highersign}
Let $M$ be a closed connected oriented manifold $M$, and let $\pi$ be
a countable discrete group {\lp}usually taken to be the fundamental group of
$M${\rp}. Let $B\pi$ be a classifying space for $\pi$, a CW complex
with contractible universal cover and fundamental group $\pi$, and let
$f\co M \to B\pi$ be a continuous map. {\lp}Up to homotopy, it's
determined by the induced homomorphism $\pi_1(M)\to \pi$.{\rp} The
associated \emph{higher signature} of $M$ is $f_*(\cL(M)\cap
[M])\in H_\bullet(B\pi, \bQ)$.
\end{definition}
\begin{boxedminipage}{\boxwidth}
\begin{conj}[Novikov's Conjecture]
\label{conj:Nov}
Any higher signature   $f_*(\cL(M)\cap [M])\in H_\bullet(B\pi, \bQ)$
is always an oriented homotopy invariant.  In other words, if
$M$ and $M'$ are closed connected oriented manifolds and if
$\varphi\co M'\to M$ is an orientation-preserving homotopy
equivalence and $f\co M\to B\pi$, then
\[
f_*(\cL(M)\cap [M]) =
(f\circ\varphi)_* (\cL(M')\cap [M'])\in H_\bullet(B\pi, \bQ).
\]
\end{conj}
\end{boxedminipage}

The utility of the conjecture can be illustrated by an example.
\begin{problem}
\label{prob:CPxS}
Classify smooth compact $5$-manifolds homotopy equivalent to
$\bC\bP^2\times S^1$. {\lp}Note: the diffeomorphism classification of smooth
$4$-manifolds homotopy equivalent to $\bC\bP^2$ is not known, since
surgery breaks down in the smooth category in dimension $4$. It is
known by work of Freedman \cite{MR679066}
that up to \emph{homeomorphism}, there are exactly
two closed topological $4$-manifolds homotopy equivalent
to $\bC\bP^2$, but for the ``exotic'' one, the product with $S^1$ does
not have a smooth structure.{\rp}
\end{problem}
\begin{proof}
Suppose $M$ is a smooth closed manifold of the homotopy type of
$\bC\bP^2\times S^1$. There is a smooth map $f\co M\to S^1$ inducing an
isomorphism on $\pi_1$, and we can take this to be the map $f\co M\to
B\pi$, $\pi=\bZ$, for the case of the conjecture proven by Novikov
himself. So the conjecture implies that if $K=f^{-1}(\pt)$, the
inverse image of a regular value of $f$, then $K$ has signature $1$.
This fixes the first Pontrjagin class of $M$.  Furthermore, $K$ being
a smooth $4$-manifold with signature $1$, it is in the same oriented
bordism class as $\bC\bP^2$. From this we can get a normal bordism $W^6$
between $M$ (with its stable normal bundle $\nu$) and $\bC\bP^2\times
S^1$ (with its stable normal bundle $\xi$). We plug into the surgery
machine and try to do surgery to convert this to an $h$-cobordism
(and thus automatically an $s$-cobordism, since $\Wh(\bZ)=0$). The
surgery obstruction lives in $L_6(\bZ[\bZ])$.  This group turns out to
be $\bZ/2$ (coming from the image of the Arf invariant in
$L_6(\bZ)\cong \bZ/2$).  So there are not a 
lot of possibilities.  In fact one can show by studying the
continuation of the sequence \eqref{eq:surgseq} to the left that $M$
is diffeomorphic to $\bC\bP^2\times S^1$.  But note that the key
ingredient in the whole argument is the Novikov Conjecture, which pins
down the first Pontrjagin class.\qed
\end{proof}  

\section{Methods of Proof}
\label{sec:methods}

Work on the Novikov Conjecture began almost as soon as the conjecture
was formulated.  Roughly speaking, methods fall into three different
categories: topological, analytic, and algebraic.  The
\emph{topological} approach began with Novikov's own work on the free
abelian case of the conjecture, which we already mentioned in the
case $\pi=\bZ$, and which only uses transversality and basic homology
theory.  This method was generalized in work of Kasparov, Farrell-Hsiang,
and Cappell \cite{MR0287565,MR0385867,MR0438349}, who used
codimension-one splitting methods to deal with free abelian and
poly-$\bZ$ groups, and certain kinds of amalgamated free products.

Subsequent topological approaches to the conjecture have been
based on \emph{controlled topology} (if you like, a blend of analysis
and topology since it amounts to topology with $\delta$-$\varepsilon$
estimates) or on various methods in stable homotopy theory.
There is a lot more in this area than we can possibly summarize here,
but it is discussed in detail in \cite{MR1388295}, which includes a
long bibliography.

The \emph{analytic} approach began with the important contribution of
Lusztig \cite{MR0322889}.  The key idea here is to realize the higher
signature of Definition \ref{def:highersign} as the index of a family
of elliptic operators, just as Atiyah and Singer \cite[\S6]{MR0236952} had
reproven Hirzebruch's signature theorem by realizing the signature as
the index of a certain elliptic operator, now universally called the
signature operator.  (This is just the operator $d+d^*$ operating on
differential forms, but with a grading on the forms coming from the
Hodge $*$-operator.)  A major step forward from the work of Lusztig
came with the work of Mishchenko \cite{MR0362407,MR561226}
and Kasparov \cite{MR1388299,MR918241,MR1216182}, who realized that one
could generalize this construction by using ``noncommutative''
families of elliptic operators, based on a $C^*$-algebra completion
$C^*(\pi)$ of the algebraic group ring $\bC\pi$.  
Underlying this method was the idea \cite{MR561226,MR1388305} that 
because of the inclusions $\bZ\pi\hookrightarrow \bC\pi
\hookrightarrow C^*(\pi)$, there is a natural map $L_n(\bZ\pi)\to
L_n(C^*(\pi))$, and that because the spectral theorem enables one to
diagonalize quadratic forms over a $C^*$-algebra, the $L$-groups and
topological $K$-groups of a $C^*$-algebra essentially coincide.  As we
will see in the next section, the analytic approach to the Novikov
conjecture is the one that has attracted the most recent attention,
though there is still plenty of work being done on topological and
algebraic methods. 

Algebraic approaches to proving the Novikov conjecture depend on a
finer understanding of the surgery exact sequence \eqref{eq:surgseq}
and the $L$-groups.  For a homotopy equivalence of manifolds
$\varphi\co M'\to M$, the difference $\varphi_*(\cL(M')\cap [M']) -
(\cL(M)\cap [M]) \in H_\bullet(M,\bQ)$ is basically $\eta([M'\to
  M])\otimes_{\bZ} \bQ$ in \eqref{eq:surgseq}.  The Novikov conjecture
says that this should vanish when we apply $f_*$, $f\co M\to B\pi$.
Since we could also apply \eqref{eq:surgseq} with $M$ replaced by
$B\pi$ (at least if $B\pi$ can be chosen to be a manifold --- but there is
a way of getting around this), exactness in \eqref{eq:surgseq} shows
that the Novikov Conjecture is equivalent to rational injectivity of
the map $\alpha$ in \eqref{eq:surgseq}, when we replace $M$ by $B\pi$.

More precisely, we need to make use an idea of Quinn \cite{MR0282375},
that the $L$-groups are the homotopy groups of a spectrum:
\[
L_n(\bZ\pi) = \pi_n(\bL_\bullet(\bZ\pi))
\]
and that the map $\alpha$ in the surgery exact sequence
\eqref{eq:surgseq} comes from an \emph{assembly map} which is the
induced map on homotopy groups of a map of spectra
\[
A_M\co M_+ \wedge \bL_\bullet(\bZ) \to \bL_\bullet(\bZ\pi).
\]
This map factors (via $f\co M\to B\pi$) through a similar map
\begin{equation}
\label{eq:assemb}
A_\pi\co B\pi_+ \wedge \bL_\bullet(\bZ) \to \bL_\bullet(\bZ\pi).
\end{equation}
If $A_\pi$ in \eqref{eq:assemb} induces a rational injection on
homotopy groups, then 
the Novikov Conjecture follows from exactness of \eqref{eq:surgseq}.
On the other hand, if $A_\pi$ is not rationally injective, then one
can construct an $M$ and a higher signature for it that is not
homotopy invariant.  So the Novikov Conjecture is reduced to a
statement which at least in principle is purely algebraic, as Ranicki
in \cite{MR561227,MR1211640} gives a purely algebraic construction of the
surgery spectra and of the map $A_\pi$, leading to the exact sequence
\eqref{eq:algsurgseq}.\footnote{It turns out that
\eqref{eq:algsurgseq} coincides with the analogue of
\eqref{eq:surgseq} in the topological, rather than smooth, category,
but the difference between these is rather small since all homotopy
groups of $\mathit{Top}/O$ are torsion.} 

\section{Variations on a Theme}
\label{sec:variations}
One of the most interesting features of the Novikov Conjecture is that
it is closely related to a number of other useful conjectures.  Some 
of these are known to be true, some are known to be false, and most
are also unsolved.  But even the ones that are false are false for
somewhat subtle reasons, and still carry some ``element of truth.''
Here we mention a number of these related conjectures and something
about their status.

\begin{boxedminipage}{\boxwidth}
\begin{conj}[Borel's Conjecture]
\label{conj:Borel} 
Any two closed aspherical {\lp}i.e., having contractible universal
covers{\rp} manifolds $M$ and $M'$ with the same fundamental group
are homeomorphic.  In fact, any
homotopy equivalence $\varphi\co M'\to M$ of such manifolds is
homotopic to a homeomorphism.
\end{conj}
\end{boxedminipage}

This conjecture is known to have been posed informally by Armand
Borel, before the formulation of Novikov's Conjecture, and was
motivated by the Mostow Rigidity Theorem. It amounts to a kind
of topological rigidity for aspherical manifolds.  Note that if $M$ is
aspherical with fundamental group $\pi$ and $n=\dim M\ge 5$,
then we can take $M=B\pi$, and Borel's conjecture amounts to saying
that in the surgery sequence \eqref{eq:surgseq} in the topological
category, $\cS(M)$ is just a single point, or by exactness, the
assembly map $A_\pi$ is an equivalence.  This implies the Novikov
Conjecture for $\pi$, but is stronger.

Incidentally, it is known now that the analogue of Borel's Conjecture,
but with homeomorphism replaced by diffeomorphism, is false. The
simplest counterexample is with $M=T^7$, the $7$-torus.  Since a torus
is parallelizable, Wall pointed out in \cite[\S 15A]{MR1687388}
that the set of smooth structures on $T^n$ compatible
with the standard PL structure is parameterized by $[T^n, PL/O]$ (for
$n\ge5$). It is known that the classifying space $PL/O$ is
$6$-connected and that (for $j\ge 7$) its $j$-th homotopy 
group can be identified with the group $\Theta_j$ of smooth
homotopy $j$-spheres.\footnote{The group operation is the connected
  sum; inversion comes from reversing the orientation.} Since
$\Theta_7\cong \bZ/28$ by \cite{MR0082103,MR0148075}, 
the differentiable structures on $T^7$ are parameterized by
$[T^7, PL/O]\cong [T^7, K(\Theta_7, 7)] \cong
H^7(T^7, \Theta_7)\cong \bZ/28$ and there are $28$ different
differentiable structures on $T^7$. A series of counterexamples 
with negative curvature to the smooth Borel conjecture was constructed in
\cite{MR1002632,MR1203985}.

The fundamental group $\pi$ of an aspherical manifold $M$ (even if
noncompact) has to be
torsion-free, since if $g\in\pi$ has finite order $k>1$, it would act freely
on the universal cover $\wM$, and $\wM/\langle g\rangle$ would be a
finite-dimensional model for $B\bZ/k$, contradicting the fact that
$\bZ/k$ has homology in all positive odd dimensions. So Conjecture
\ref{conj:Borel} can't apply to groups with torsion. In fact, the
result of \cite{MR1988288} shows that for groups with torsion, 
$A_\pi$ in \eqref{eq:assemb}
is never an equivalence. We will come back to this shortly.

However, we have already mentioned the role of the Whitehead group,
which comes from the algebraic $K$-theory of $\bZ\pi$, in studying
manifolds with fundamental group $\pi$.  An important conjecture which
we have already mentioned is:

\begin{boxedminipage}{\boxwidth}
\begin{conj}[Vanishing of Whitehead Groups] 
\label{conj:Wh}
If $\pi$ is torsion-free, then $\Wh(\pi)=0$.
\end{conj}
\end{boxedminipage} 

Note that if Conjecture \ref{conj:Wh} fails and $\pi$ is the
fundamental group of a closed manifold $M$, then by the $s$-cobordism
theorem, there is an $h$-cobordism $W$ with $\partial W = M \sqcup
(-M')$ which is not a product,
and we have a homotopy equivalence $M'\to M$ which is not simple,
hence Borel's Conjecture, Conjecture \ref{conj:Borel}, fails for $M$.

More generally, one can ask what one can say about the algebraic
$K$-theory of $\bZ\pi$ in all degrees.  Loday \cite{MR0447373}
constructed an assembly map $B\pi_+\wedge \bK(\bZ)\to \bK(\bZ\pi)$,
and this being an equivalence would say that all of the algebraic
$K$-theory of $\bZ\pi$ comes in some sense from homology of $\pi$ and
$K$-theory of $\bZ$.  This is known in some cases --- for $\pi$ free
abelian, it follows from the ``Fundamental Theorem of $K$-theory.''
The assembly map being an equivalence in degrees $\le 1$ for
torsion-free groups $\pi$ and $R=\bZ$ implies Conjecture
\ref{conj:Wh}. The analogue of Novikov's Conjecture for $K$-theory is

\begin{boxedminipage}{\boxwidth}
\begin{conj}[Novikov Conjecture for $K$-Theory]
\label{conj:KNov}
Let $R=\bZ,\bQ,\bR,\text{ or }\bC$ and let $\pi$ be a discrete group.
Then the assembly map $B\pi_+\wedge \bK(R)\to \bK(R\pi)$ induces
an injection of rational homotopy groups.
\end{conj}
\end{boxedminipage} 

Conjecture \ref{conj:KNov} was proved (with $R=\bZ$, the most
important case) for groups $\pi$ with finitely
generated homology in \cite{MR1202133}.  It was also proved (without
rationalizing) in \cite{MR1388308}, when $\pi$ is a discrete, cocompact,
torsion-free discrete subgroup of a connected Lie group.
Subsequently, Carlsson and Pedersen \cite{MR1341817}
proved it (without rationalizing) for any group $\pi$ for which there
is a finite model for $B\pi$, such that the universal cover $E\pi$ of
$B\pi$ admits a contractible metrizable $\pi$-equivariant
compactification $X$ such that compact subsets of $E\pi$ become small
near the ``boundary'' $X\smallsetminus E\pi$.  This was recently
improved \cite{MR3259041} to the case where there
is a finite model for $B\pi$ and $\pi$ has finite decomposition
complexity, which is a tameness condition on $\pi$ viewed as a metric
space with the word length metric (for some finite generating set).

As we have already mentioned, for groups with torsion, 
the assembly map $A_\pi$ of \eqref{eq:assemb} is never an
equivalence. For similar reasons, one also can't expect the $K$-theory
assembly map to be an equivalence for groups with torsion. The correct
replacement seems to be the following.\footnote{Just for
the experts: one needs to use the $-\infty$ decoration on the
$L$-spectra here.}

\begin{boxedminipage}{\boxwidth}
\begin{conj}[Farrell-Jones Conjecture]
\label{conj:FJ}
Let $\pi$ be a discrete group and let $\cF$ be its family of virtually
cyclic subgroups {\lp}subgroups that contain a cyclic subgroup of
finite index{\rp}.  Such subgroups are either finite or else admit a
surjection with finite kernel onto either $\bZ$ or the infinite
dihedral group $(\bZ/2)*(\bZ/2)$. Let $E_{\cF}(\pi)$ denote the
universal $\pi$-space with isotropy in $\cF$. This is a
contractible $\pi$-CW-complex $X$ with all isotropy groups in $\cF$
{\lp}for the $\pi$-action{\rp} and with $X^H$ contractible for each $H\in 
\cF$. It is known to be uniquely defined up to $\pi$-homotopy
equivalence.  Then the assembly maps
\begin{equation}
\label{eq:FJassemb}
H_\bullet^\pi(E_{\cF}(\pi); \bL(\bZ))\to \bL(\bZ\pi) \quad
\text{and}\quad
H_\bullet^\pi(E_{\cF}(\pi); \bK(R))\to \bK(R\pi)
\end{equation}
are isomorphisms for $R=\bZ,\bQ,\bR,\text{ or }\bC$.
\end{conj}
\end{boxedminipage} 

When $\pi$ is torsion-free, \eqref{eq:FJassemb} is just the assembly
map \eqref{eq:assemb} or its $K$-theory version, and the conjecture
says that the assembly map is an equivalence. Conjecture \ref{conj:FJ}
implies Conjectures \ref{conj:Borel}, \ref{conj:Nov}, and \ref{conj:KNov},
even for groups with torsion, as well as Conjecture \ref{conj:Wh}.
More details on Conjecture
\ref{conj:FJ} may be found in \cite{MR2827832}, in
\cite[Ch.\ 19--24]{MR2117411}, or in \cite{MR2181833}.
The $K$-theory version of the conjecture has been proven in
\cite{MR2138135} for fundamental groups of manifolds of negative
curvature and in \cite{MR2385666} for hyperbolic groups, and both the
$K$-theory and $L$-theory versions have been proven for
certain groups acting on trees in \cite{MR2210223,MR2954664} and for
cocompact lattice subgroups of Lie groups in \cite{MR3164984}.
Rational injectivity of \eqref{eq:FJassemb} holds under much weaker
conditions; see for example \cite{MR3178242}.

Another variation on the Novikov Conjecture is to consider the
situation where a finite group $G$ acts on a manifold, and one wants
to study $G$-equivariant invariants of $M$. Under suitable circumstances,
one finds that the fundamental group of $M$ leads to a certain extra amount
of equivariant topological rigidity.  To formulate the analogue of
Conjecture \ref{conj:Nov}, one needs a substitute for the homology
$L$-class $\cL(M)\cap [M]$.  The easiest way to formulate this is in
$K$-homology, since Kasparov \cite{MR0488027,MR582160}, following ideas
of Atiyah and Singer, showed that an elliptic differential operator $D$ on
$M$ naturally leads to a $K$-homology class $[D]\in K_\bullet(M)$
(see also \cite{MR1817560} for an exposition), and when $D$ is
$G$-invariant, the class naturally lives in  $K_\bullet^G(M)$. The
image of $[D]$ in $K_\bullet^G(\pt)=R(G)$ under the map induced by $M\to \pt$
is the equivariant index $\inde_G D\in R(G)$ in the sense of  Atiyah
and Singer. When $D$ is the signature operator, $\cL(M)\cap [M]$ is
basically (except for some powers of $2$, not important here) the
Chern character of $[D]\in K_\bullet(M)$, and so if $f\co M\to B\pi$,
the higher signature of Definition \ref{def:highersign}, is basically
the Chern character of $f_*([D])$. That motivates the following.

\begin{boxedminipage}{\boxwidth}
\begin{conj}[Equivariant Novikov Conjecture \cite{MR1076524}]
\label{conj:equivNov}
Let $M$ be a closed oriented manifold admitting an action of a finite
group $G$, and suppose $f\co M\to X$ is a $G$-equivariant smooth map
to a finite $G$-CW complex which is $G$-equivariantly 
aspherical {\lp}i.e., $X^H$ is aspherical for all subgroups $H$ of
$G${\rp}. Let $\varphi\co M'\to M$ be a $G$-equivariant map of closed
$G$-manifolds which, non-equivariantly, is a homotopy equivalence.
Then if $[D_M]$ and $[D_{M'}]$ denote the equivariant $K$-homology
classes of the signature operators on $M$ and $M'$, respectively,
\[
f_*([D_M])= (f\circ \varphi)_*([D_{M'}]) \in K_\bullet^G(X).
\]
\end{conj}
\end{boxedminipage}

Various generalizations and applications to rigidity theorems
are possible (see for example
\cite{MR951234,MR982331}), but we won't go into details 
here.  Conjecture \ref{conj:FJ} was proven in \cite{MR1076524} for 
$X$ is a closed manifold of nonpositive curvature and in
\cite{MR1433118} for $X$ a Euclidean building, in both cases with $G$
acting by isometries.

\section{New Directions}
\label{sec:directions}

The conjectures we discussed in Section \ref{sec:variations} are
fairly directly linked to the original Novikov Conjecture, and it is
easy to see how they are connected with topological rigidity of highly
connected manifolds.  But in this section, we will discuss a number of
other conjectures which grew out of work on Novikov's Conjecture but
which go somewhat further afield, to the point where the connection
with the original conjecture may not be immediately obvious.  However,
we will try to explain the relationships as we go along.

We have already mentioned the assembly map and the Farrell-Jones
Conjecture (Conjecture \ref{conj:FJ}), which gives a conjectural
calculation of the $L$-groups $L_\bullet(\bZ\pi)$ for a discrete group
$\pi$. However, work on Novikov's Conjecture by analytic techniques
(see Section \ref{sec:methods}) already required passing from the integral
group ring to the complex group ring (this only affects $2$-torsion in
the $L$-groups) and then completing $\bC\pi$ to a $C^*$-algebra.
For $C^*$-algebras, $L$-theory is basically the same as topological
$K$-theory, and even for real $C^*$-algebras, they agree after
inverting $2$ \cite[Theorem 1.11]{MR1388305}.  So it's natural to ask
if assembly can be used to compute the topological $K$-theory of
$C^*(\pi)$. For the full group $C^*$-algebra this seems to be
impossible, but for the \emph{reduced} group $C^*$-algebra 
$C^*_r(\pi)$,\footnote{It is known that the natural map
$C^*(\pi)\twoheadrightarrow C^*_r(\pi)$ is an isomorphism if and only
  if $\pi$ is amenable.} the
completion of $\bC\pi$ for its action on $L^2(\pi)$, there is a good
guess for a purely topological calculation of
$K_\bullet(C^*_r(\pi))$.  (Here $K_\bullet$ denotes \emph{topological}
$K$-theory for Banach algebras, which satisfies Bott periodicity.
This is much more closely related to $L$-theory, which is
$4$-periodic, than is algebraic $K$-theory in the sense of Quillen.)
This guess is given by the \emph{Baum-Connes Conjecture}, originally
formulated in \cite{MR1769535,MR928402} and further refined in
\cite{MR1292018}. (See also \cite{MR1648112} for a nice quick survey.)
The conjecture applies to far more than just
discrete groups; it applies to locally compact groups, to such groups
``with coefficients'' (i.e., acting on a $C^*$-algebra), and even to
groupoids \cite{MR1798599}.
In its greatest generality the conjecture is known to be
false \cite{MR1911663}, though a patch which might repair it has been
proposed \cite{2013arXiv1311.2343B}.  However, the original version of
the conjecture is still open, though the literature on the conjecture
has grown to more than 300 items.  To avoid having to talk about
Kasparov's $KK$-theory, we will omit discussion of the conjecture with
coefficients, and will just stick to the original conjecture for
groups.

\begin{boxedminipage}{\boxwidth}
\begin{conj}[Baum-Connes Conjecture]
\label{conj:BC}
Let $G$ be a second countable locally compact group, and let
$C^*_r(G)$ denote the completion of $L^1(G)$ for its action by left
convolution on $L^2(G)$. Then there is a natural assembly map
\[
\mu\co K_\bullet^G(\cE G)\to K_\bullet(C^*_r(G)),
\]
where $\cE G$ is the universal proper $G$-space {\lp}a contractible
space on which $G$ acts properly{\rp}, and this map is an isomorphism.
If $G$ has no nontrivial compact subgroups, then the assembly map
simplifies to
\[
\mu\co K_\bullet(B G)\to K_\bullet(C^*_r(G)).
\]
\end{conj}
\end{boxedminipage} 

\begin{proposition}
\label{prop:BCgivesNC}
Conjecture \ref{conj:BC} implies Conjecture
\ref{conj:Nov}.  
\end{proposition}
\begin{proof}
For this we take $G=\pi$ to be discrete and
countable. For simplicity, we also work with the periodic $L$-theory
spectra instead of the connective ones.  (The difference only affects
the bottom of the surgery sequence \eqref{eq:surgseq}.)
If $\pi$ is torsion-free, the domain of $\mu$ is
$K_\bullet(B\pi) = H_\bullet (B\pi; \bK^{\text{top}})$. But after
inverting $2$, $\bK^{\text{top}}$ is just a direct sum of two copies
of $\bL(\bZ)$, one of them shifted in degree by $2$. So if
Conjecture \ref{conj:BC} holds for $\pi$ and $\pi$ is torsion-free, we
have the commuting diagram
\begin{equation}
  \label{eq:assemblychase}
\xymatrix@R-1pc{H_\bullet (B\pi; \bL(\bZ))\otimes \bQ \ar[r]^(.55){A_\pi}
  \ar@{^{(}->}[dd]& 
  L_\bullet(\bZ\pi) \otimes \bQ \ar[d]\\
  & L_\bullet(C^*_r(\pi)) \otimes \bQ \ar[d]^\cong\\
  H_\bullet (B\pi; \bK^{\text{top}})\otimes \bQ \ar[r]_\cong^\mu &
  K_\bullet(C^*_r(\pi)) \otimes \bQ.}
\end{equation}
Diagram \eqref{eq:assemblychase} immediately implies that the
rational $L$-theory assembly map $A_\pi$ (the same map as the
map induced on rational homotopy groups by \eqref{eq:assemb}) is 
injective.

If $\pi$ is not torsion-free, then $\cE\pi$ and $E\pi$ are not the
same,\footnote{In the extreme case where $\pi$ is a torsion group,
  $\cE\pi=\pt$, while if $\pi$ is nontrivial, $E\pi$ is necessarily
  infinite dimensional.} 
but there is always a $\pi$-equivariant map $E\pi\to \cE\pi$.  Thus
we need only replace \eqref{eq:assemblychase}
by the diagram
\begin{equation}
  \label{eq:assemblychase1}
  \xymatrix{H_\bullet (B\pi; \bL(\bZ))\otimes \bQ
    \ar[r]\ar@{^{(}->}[dd] \ar@/^1pc/[rr]^{A_\pi}&
    H_\bullet^\pi(\cE\pi; \bL(\bZ)) \otimes \bQ \ar[r]
    \ar@{^{(}->}[dd]  &   L_\bullet(\bZ\pi) \otimes \bQ \ar[d]\\
  & & L_\bullet(C^*_r(\pi)) \otimes \bQ \ar[d]^\cong\\
  H_\bullet(B\pi;\bK^{\text{top}}) \otimes \bQ \ar[r]^\alpha &
  H_\bullet^\pi(\cE\pi;\bK^{\text{top}}) \otimes \bQ \ar[r]_\cong^\mu &
  K_\bullet(C^*_r(\pi)) \otimes \bQ.}
\end{equation}
Since points in $\cE\pi$ have finite isotropy, and since the $\pi$-map
$\pi \twoheadrightarrow \pi/\sigma$, $\sigma$ a finite subgroup of
$\pi$, induces the map $\bZ\hookrightarrow R(\sigma)$ on equivariant
$K$-homology, a spectral sequence argument shows that
the bottom left map $\alpha$ in
\eqref{eq:assemblychase1} is injective, and so by a diagram chase, $A_\pi$
is injective. \qed
\end{proof}
Thus Conjecture \ref{conj:BC} (for the case of discrete groups)
implies Conjecture \ref{conj:Nov}.  However, Conjecture \ref{conj:BC}
for \emph{non-discrete} groups is also quite interesting and
important.  There are two main reasons for this:
\begin{enumerate}
\item There are ``change of group methods'' that enable one to pass
  from results for a group to results for a closed subgroup.  Many of
  the significant early results on Novikov's Conjecture were proved by
  considering discrete groups $\pi$ that embed in a Lie group (or
  $p$-adic Lie group) and then using these change of group methods to
  pass from the Lie group to the discrete subgroup.
\item The Baum-Connes Conjecture for connected Lie groups (also known as the
  Connes-Kasparov Conjecture) and the same conjecture
  for $p$-adic groups are both quite
  interesting in their own right, and say a lot about representation
  theory. For an introduction to this topic, see \cite{MR1292018,MR1648112}.
  For some of the more significant results, see
  \cite{MR894996,MR1914617,MR1467072,MR2539769}. For recent applications to
  harmonic analysis on reductive groups, see 
  \cite{MR1941993,MR2314099,MR3271238,2015arXiv150504091R}.
\end{enumerate}

Another direction arising out of both the controlled topology and the
analytic approaches to Novikov's Conjecture leads to the so-called
\emph{coarse Baum-Connes Conjecture}
\cite{MR1147350,MR1344138,MR1388312}.  
This conjecture deals with the large-scale geometry of metric  spaces $X$ of
bounded geometry (think of complete Riemannian manifolds with
curvature bounds, or of finitely generated groups with a word-length
metric).  Roughly speaking, the coarse Novikov Conjecture
says that indices of generalized elliptic operators capture all of the
coarse (i.e., ``large-scale'') rational homology of such a space $X$.

\begin{boxedminipage}{\boxwidth}
\begin{conj}[Coarse Baum-Connes and Novikov]
\label{conj:coarseBC}
Let $X$ be a uniformly contractible
locally compact complete metric space of bounded
geometry, in which all metric balls are compact. Let $KX_\bullet(X)$
be the coarse $K$-homology of $X$ {\lp}the direct limit of the
$K$-homologies of successively coarser Rips complexes{\rp} and let
$C^*(X)$ be the $C^*$-algebra of locally compact, finite propagation
operators on $X$.  Then Roe defined a natural assembly map
\begin{equation}
\label{eq:coarseassmb}
\mu\co KX_\bullet(X)\to K_*(C^*(X)).
\end{equation}
The coarse Baum-Connes Conjecture is that $\mu$ is an isomorphism; the
coarse Novikov Conjecture is that $\mu$ is rationally injective.
\end{conj}
\end{boxedminipage} 

Positive results on Conjecture \ref{conj:coarseBC} may be found in
\cite{MR1147350,MR1344138,MR1388312,MR3116568,MR3325537,MR2891024,MR3266245,MR3345182}.  

However, it is known that the conjecture fails in various situations
\cite{MR1983785,MR1728880,MR1911663},
especially if one drops the bounded geometry assumption.

The coarse Baum-Connes conjecture implies the Novikov conjecture under
mild conditions.  To see this, suppose for example that there is a
compact metrizable model $Y$ for $B\pi$, and let $X=E\pi$ be its
universal covering. Then there is a commutative diagram
\[
\xymatrix{K_*(B\pi) \ar[r]^\alpha \ar[d]^{\text{tr}}_{\cong}& K_*(C^*_r(\pi))
  \ar[d]^{\text{tr}}\\
\pi_*(\bK X_*(X)^{h\pi}) \ar[r]^{\mu^{h\pi}} & \pi_*(\bK_*(C^*(X))^{h\pi}),}
\]
where $\alpha$ is usual Baum-Connes assembly, $\mu$ is as in
Conjecture \ref{conj:coarseBC}, $h\pi$ denotes homotopy fixed points,
and $\text{tr}$ is a transfer map. Then $\mu$ being an isomorphism implies
that $\mu^{h\pi}$ is an isomorphism, and so we get a splitting for
$\alpha$.  Refinements of this argument, as well as generalizations of
the coarse Baum-Connes conjecture, may be found in \cite{MR2748337}.

Thinking of $C^*_r(\pi)$ as being (up to Morita equivalence) the same
thing as the fixed points of $\pi$ on $C^*(X)$ also gives rise to a
nice way of relating the surgery exact sequence \eqref{conj:coarseBC}
to the Baum-Connes assembly map.  This was accomplished in the series
of papers \cite{MR2220522,MR2220523,MR2220524,2013arXiv1309.4370P},
which set up a natural transformation from the surgery sequence to a
long exact sequence where the $C^*$-algebraic assembly map corresponds
to the $L$-theory assembly map in the original sequence.  This gives
an even more direct connection between coarse Baum-Connes and surgery
theory.

Other ``new directions'' from Novikov's Conjecture arise from
replacing the higher signature of Definition \ref{def:highersign}
with other sorts of ``higher indices.''  For example, an important
case is obtained by replacing $\cL(M)$ with $\wcA(M)$, the total
$\widehat A$ class.  This is again a certain polynomial in the
rational Pontrjagin class, and has the property that when $M$ is a
spin manifold, $\wcA(M)\cap [M]$ is the Chern character of the class
$[D]$ defined by the Dirac operator on $M$.  (Here the reader doesn't
need to know much about the Dirac operator $D$ except for the fact that
it's an elliptic first-order differential operator canonically defined
on a Riemannian manifold with a spin structure.)
It was pointed out by Lichnerowicz \cite{MR0156292} that when $M$
is closed and has positive scalar curvature, then the spectrum of $D$
must be bounded away from $0$, and thus $\inde(D) = \langle \wcA(M),
[M]\rangle$ has to vanish.  When $M$ is not simply connected, a major
strengthening of this is possible:

\begin{boxedminipage}{\boxwidth}  
\begin{conj}[Gromov-Lawson Conjecture \cite{MR720933}]
\label{conj:GL}
Let $M$ be a connected closed spin Riemannian manifold of positive
scalar curvature, let $\pi$ be a discrete
group, and let $f\co M\to B\pi$ be a continuous map {\lp}determined up to
homotopy by a homomorphism $\pi_1(M)\to \pi${\rp}. Then the
higher $\widehat A$-genus $f_*(\wcA(M)\cap [M])\in
H_\bullet(B\pi,\bQ)$ vanishes. 
\end{conj}
\end{boxedminipage} 

This conjecture is still open in general, but it is known to be
closely related to Novikov's Conjecture. For example, it was shown in
\cite{MR720934} that Conjecture \ref{conj:GL} is true whenever the
$K$-theory assembly map $K_\bullet(B\pi)\to K_\bullet(C^*_r(\pi))$ is
rationally injective, and thus \emph{a fortiori} whenever Conjecture
\ref{conj:BC} holds.  It also can be deduced from certain cases of
Conjecture \ref{conj:coarseBC}, by a descent argument similar to the
one above.  The Lichnerowicz argument also  applies to complete noncompact
spin manifolds $M$ of \emph{uniformly} positive scalar curvature, and when
Conjecture \ref{conj:coarseBC} holds, one gets obstructions to existence of
such metrics living in $K_\bullet(C^*(X))$ whenever there is a coarse
map $M\to X$.

Conjecture \ref{conj:GL} can be refined to conjectures about
necessary and sufficient conditions for positive scalar curvature.
Here we just mention a few of several possible versions.  For these it's
necessary to go beyond ordinary homology and to consider
$KO$-homology, the homology theory dual to the (topological)
$K$-theory of real vector bundles.  This theory is $8$-periodic
and has coefficient groups $KO_j = \bZ$ when $j$ is divisible by $4$
(this part is detected by the Chern character to ordinary homology),
$\bZ/2$ when $j\equiv 1,2 \pmod 8$, $0$ otherwise. The class $[D]$ of
the Dirac operator on a spin manifold $M$ lives in $KO_n(M)$, $n=\dim
(M)$. While the actual operator $D$ depends on a choice of a Riemannian
metric, the class $[D]\in KO_n(M)$ does not, so that the following
conjecture makes sense.

\begin{boxedminipage}{\boxwidth}
\begin{conj}[Gromov-Lawson-Rosenberg Conjecture]
\label{conj:GLR}
Let $M$ be a connected closed spin manifold with fundamental group
$\pi$ and Dirac operator $D_M$,
and let $f\co M\to B\pi$ be the classifying map for the
universal cover. Let $A\co KO_\bullet(B\pi)\to KO_\bullet(C^*_r(\pi))$
be the assembly map in real $K$-theory.  Then $M$ admits a Riemannian
metric of positive scalar curvature if and only if
$A\circ f_*([D_M])= 0 $ in $KO_n(C^*_r(\pi))$, $n=\dim M\ge 5$.
\end{conj}
\end{boxedminipage}

The restriction to $n\ge 5$ is
needed only to use surgery methods to construct a metric of positive
scalar curvature when the obstruction vanishes; it is not needed to
show that there is a genuine 
obstruction to positive scalar curvature when $A \circ f_* ([D_M])\ne
0$, which was proven in \cite{MR842428}. 
For the next conjecture, we need to introduce a choice
of \emph{Bott manifold}, a geometric representative for Bott
periodicity in $KO$-homology. This is a simply connected closed spin manifold
$\text{Bt}^8$ of dimension $8$ with $\langle \wcA(\text{Bt}^8),
[\text{Bt}^8]\rangle =1$. It may be chosen to be Ricci flat.

\begin{boxedminipage}{\boxwidth}
\begin{conj}[Stable Gromov-Lawson-Rosenberg Conjecture]
\label{conj:sGLR}
Let $M$ be a connected closed spin manifold  with fundamental group
$\pi$ and Dirac operator $D_M$,
and let $f\co M\to B\pi$ be the classifying map for the
universal cover. Let $\text{Bt}^8$ be a Bott manifold as above.
Then $M$ {\bfseries stably} admits a Riemannian metric of positive scalar
curvature, in the sense that $M\times \overbrace{\text{Bt}^8\times
  \cdots \times \text{Bt}^8}^k$ admits such a metric for some $k$,
if and only if
$A\circ f_*([D_M])= 0 $ in $KO_n(C^*_r(\pi))$, $n=\dim M$.
\end{conj}
\end{boxedminipage} 

There are simple implications
\[
\text{Conj. \ref{conj:GLR}} \Rightarrow
\text{Conj. \ref{conj:sGLR}},\quad
\text{Conj. \ref{conj:sGLR}} + \text{injectivity of }A\Rightarrow
\text{Conj. \ref{conj:GL}}.
\]
The (very strong) Conjecture \ref{conj:GLR} is known to hold for
especially nice groups, such as free abelian groups \cite{MR842428},
hyperbolic groups of low dimension \cite{MR1778107},
and finite groups with periodic cohomology \cite{MR1484887}, but it fails
in general \cite{MR1778107,MR1632971}. Conjecture  \ref{conj:sGLR} is
weaker, and holds for all the known counterexamples to Conjecture
\ref{conj:GLR}. It was formulated and proven for finite groups in
\cite{MR1321004}. Subsequently, Stolz [unpublished] showed that it
follows from the Baum-Connes Conjecture, Conjecture \ref{conj:BC}.
For a survey on this entire field, see \cite{MR2408269}.

The last ``new direction'' we would like to discuss here comes from
replacing the higher signature in Novikov's Conjecture by the higher
Todd genus or the higher elliptic genus.  This seems to be quite
relevant for understanding the interaction between topological
invariants and algebraic geometry invariants for algebraic varieties
defined over $\bC$.

The Todd class $\cT(M)$ is still another polynomial in characteristic
classes, this time the rational Chern classes of a complex (or almost
complex) manifold. Suppose for simplicity that $M$ is a smooth
projective variety over $\bC$, viewed as a complex manifold via an
embedding into some complex projective space. The Hirzebruch
Riemann-Roch Theorem then says that
\begin{equation}
\label{eq:Hirz}  
\langle \cT(M), [M]\rangle = \chi(M, \cO_M)=
\sum_{j=0}^n  (-1)^j \dim H^j(M, \cO_M),
\end{equation}
where $\cO_M$ is the structure sheaf of $M$, the sheaf of germs of
holomorphic functions, and $n$ is the complex dimension of $M$.
The right-hand side of \eqref{eq:Hirz} is
called the \emph{arithmetic genus}.  (The original definition of the
latter by algebraic geometers like Severi
turned out to be $(-1)^n(\chi(M,\cO_M)-1)$, but the
normalization here is a bit more convenient.)  The left-hand side of
\eqref{eq:Hirz} is called the \emph{Todd genus}, and is known
to be a birational invariant.\footnote{Recall that two varieties are
  said to be birationally equivalent if there are rational maps
  between them which are inverses of each.  Since rational maps do not
  have to be everywhere defined (this is why we denote rational maps
  below by dotted lines), two varieties are birationally equivalent if
  and only if they have Zariski-open subsets which are isomorphic as
  varieties.}
Once again, if one
has a map $f\co M\to B\pi$, then we can define the associated
\emph{higher Todd genus} as $f_*(\cT(M)\cap [M])\in H_\bullet(B\pi,
\bQ)$.

\begin{boxedminipage}{\boxwidth}
\begin{conj}[Algebraic Geometry Novikov Conjecture \cite{MR2342008}]
\label{conj:higherTodd}
Let $M$ be a smooth complex projective variety, and let $f\co M\to
B\pi$ be a continuous map {\lp}for the topology of $M$ as a complex
manifold{\rp}. Let $\xymatrix{M' \ar@{.>}[r]^\varphi & M}$ be a birational
map.  Then the corresponding higher Todd genera agree, i.e.,
\[
f_*(\cT(M)\cap [M]) = (f\circ \varphi)_*(\cT(M')\cap [M'])\in
H_\bullet(B\pi, \bQ).
\]
\end{conj}
\end{boxedminipage}

Note the obvious similarity with Conjecture \ref{conj:Nov}.  However,
unlike Novikov's original conjecture, this statement is actually a
\emph{theorem} \cite{MR2282420,MR2646988}.  That follows from the fact
that if $\xymatrix{M' \ar@{.>}[r]^\varphi & M}$ is a birational
map, then $\varphi_*([D_{M'}])=[D_M] \in K_0(M)$, where $[D_M]$ denotes the
$K$-homology class of the Dolbeault operator, whose Chern character is
$\cT(M)\cap [M]$.\footnote{It takes a bit of work to make sense of
$\varphi_*$ here, since $\varphi$ may not be everywhere defined, but this can be
  done. The point is that by the factorization theorem for birational
  maps \cite{MR1896232}, we can factor $\varphi$ into a sequence of blow-ups
and blow-downs, and $\varphi_*$ is clearly defined for a blow-down (since
it is a continuous map) and is an isomorphism in this case
by the Baum-Fulton-MacPherson variant of
Grothendieck-Riemann-Roch \cite{MR549773}. In the case of a blow-up,
let $\varphi_*$ be given by the inverse of the map induced by the reverse
blow-down.}  The corresponding 
statement for the signature operator is 
\emph{not} true; a homotopy equivalence does not have to preserve the
class of the signature operator. (However, the mod $8$ reduction of
this class \emph{is} preserved \cite{MR2170494}.)

However, there is another similarity with Novikov's Conjecture  which
is pointed out in \cite{MR2342008}.  By \cite[Th\'eor\`eme
  IV.17]{MR0061823}, $\Omega_\bullet$, the graded ring of cobordism
classes of oriented manifolds, is, after tensoring with $\bQ$, a
polynomial ring in the classes of the complex projective spaces
$\bC\bP^{2k}$, $k\in \bN$.  Then if $I_\bullet$ is the ideal in
$\Omega_\bullet$ generated by all $[M]-[M']$ with $M$ and $M'$
homotopy equivalent (in a way preserving orientation), Kahn
\cite{MR0172306} proved that $\Omega_\bullet/I_\bullet \cong \bQ$,
with the quotient map identified with the Hirzebruch signature.
Similarly, $\Omega^U_\bullet$, the graded ring of cobordism
classes of almost complex manifolds, is, after tensoring with $\bQ$, a
polynomial ring in the classes of all complex projective spaces, and
the quotient of $\Omega^U_\bullet$
by the ideal generated by all $[M]-[M']$ with $M$ and
$M'$ birationally equivalent smooth projective varieties is again
$\bQ$, this time with the quotient map identifiable with the Todd genus.

These results effectively say that, up to multiples, the signature is
the only homotopy-invariant genus on oriented manifolds, and the
arithmetic genus is the only birationally invariant genus on smooth
projective varieties.  But if one considers manifolds with large
fundamental group, the situation changes. By \cite[Theorem
  4.1]{MR2342008}, a linear functional on $\Omega_\bullet(B\pi)\otimes \bQ$
that is an oriented homotopy invariant must come from the higher
signature, and by \cite[Theorem 4.3]{MR2342008}, a linear functional
on $\Omega^U_\bullet(B\pi)\otimes \bQ$ that is a birational
invariant must (under a certain technical condition satisfied in many
cases) come from the higher Todd genus.

Finally, the papers \cite{MR1400287,MR2407227,MR2545873} consider still more
analogues of higher genera with the Todd genus replaced by the
elliptic genus.  The result of \cite{MR2407227} is particularly nice;
it is the exact analogue of Conjecture \ref{conj:higherTodd}, but with
the Todd genus replaced by the elliptic genus and with birational
equivalence replaced by $K$-equivalence (a birational equivalence
preserving canonical bundles).

\bibliographystyle{spmpsci} \bibliography{NovConj} 
\end{document}